\documentclass[letterpaper, 10 pt, journal, twoside]{IEEEtran}
\usepackage{cite}
\usepackage{amsmath,amssymb,amsfonts}
\usepackage{algorithmic}
\usepackage{graphicx}
\usepackage{textcomp}
\usepackage{amsthm}

\usepackage{tikz}
\usetikzlibrary{shapes,arrows}
\usepackage{nicefrac}
\usepackage{mathtools, cuted}
\usepackage{lipsum}
\usepackage{float}
\usepackage{afterpage}
\usepackage{placeins}
\usepackage{subcaption}
\usepackage{tabularx}
\usepackage{balance}
\usepackage{textcomp}%
\usepackage{booktabs}
\usepackage{multirow}
\usepackage{accents}
\usetikzlibrary{patterns,shapes,arrows}
\usetikzlibrary{positioning}
\allowdisplaybreaks

\tikzstyle{decision} = [diamond, draw, fill=blue!20,
    text width=7em, text badly centered, inner sep=0pt]
\tikzstyle{decision1} = [diamond, draw, fill=blue!20,
    text width=7em, text badly centered, inner sep=0pt, node distance = 2.2cm]
\tikzstyle{block} = [rectangle, draw, fill=blue!20,
    text width=25em, text centered, rounded corners, minimum height=3em]
\tikzstyle{blockc1} = [rectangle, draw, fill=blue!20,
    text width=16em, text centered, rounded corners, minimum height=3em, node distance = 1.5cm]
\tikzstyle{blockc2} = [rectangle, draw, fill=blue!20,
    text width=13em, text centered, rounded corners, minimum height=3em, node distance = 1.5cm]
\tikzstyle{blockc3} = [rectangle, draw, fill=blue!20,
    text width=15em, text centered, rounded corners, minimum height=3em]
\tikzstyle{blockl2} = [rectangle, draw, fill=blue!20,
    text width=5em, text centered, rounded corners, minimum height=3em, node distance = 4cm]
\tikzstyle{blockl4} = [rectangle, draw, fill=blue!20,
    text width=5em, text centered, rounded corners, minimum height=3em, node distance = 4cm]
\tikzstyle{blockr1} = [rectangle, draw, fill=red!20,
    text width=5em, text centered, rounded corners, minimum height=3em, node distance = 4cm]
\tikzstyle{line} = [draw, very thick, color=black!50, -latex']
\tikzstyle{cloud} = [draw, ellipse,fill=red!20, node distance=2.5cm, minimum height=2em]

\newtheorem{proposition}{Proposition}

\newcommand\copyrighttext{%
  \footnotesize \textcopyright 2022 IEEE. Personal use of this material is permitted. Permission from IEEE must be obtained for all other uses, in any current or future media, including reprinting/republishing this material for advertising or promotional purposes, creating new collective works, for resale or redistribution to servers or lists, or reuse of any copyrighted component of this work in other works.}
\newcommand\copyrightnotice{%
\begin{tikzpicture}[remember picture,overlay]
\node[anchor=south,yshift=10pt] at (current page.south) {\fbox{\parbox{\dimexpr\textwidth-\fboxsep-\fboxrule\relax}{\copyrighttext}}};
\end{tikzpicture}%
}

\begin{document}
\title{Solving Dynamic Optimization Problems to a Specified Accuracy: An Alternating Approach using Integrated Residuals}

\author{Yuanbo Nie and Eric C. Kerrigan 
\thanks{Yuanbo Nie and Eric C. Kerrigan are with the Department of Aeronautics, Imperial College London, SW7~2AZ, U.K. {\tt\small yn15@ic.ac.uk}, {\tt\small 
e.kerrigan@imperial.ac.uk}}%
\thanks{Eric C. Kerrigan is also with the Department of Electrical \& Electronic Engineering, Imperial College London, London SW7~2AZ, U.K.}%
\thanks{Accepted version to be published in: IEEE Transactions on Automatic Control}%
}

\maketitle
\copyrightnotice

\begin{abstract}
We propose a novel direct transcription and solution method for solving nonlinear, continuous-time dynamic optimization problems. Instead of forcing the dynamic constraints to be satisfied only at a selected number of points as in direct collocation, the new approach alternates between minimizing and constraining the squared norm of the dynamic constraint residuals integrated along the whole solution trajectories. As a result, the method can 1) obtain solutions of higher accuracy for the same mesh compared to direct collocation methods, 2)  enables a flexible trade-off between solution accuracy and optimality, 3) provides reliable solutions for challenging problems, including those with singular arcs and high-index differential algebraic equations. 
\end{abstract}

\section{Introduction}
Model predictive control (MPC) is a modern control design method that requires the formulation and solution of a sequence of constrained dynamic optimization problems (DOPs), which arise from having to solve certain system identification, estimation and optimal control problems. For nonlinear variants of MPC (NMPC), numerically solving the DOPs often requires the use of a discretization method to transcribe the DOP into nonlinear programming problems (NLPs). To date, the common practice for NMPC implementations is to make use of existing direct transcription methods, such as direct collocation~\cite{Rawlings2020MPC}. 

Collocation belongs to a broader family of weighted residual methods, which are commonly found in the literature on the numerical solution of ordinary differential equations (ODEs) or differential algebraic equations (DAEs)~\cite{ascher1978}. Other variants of such methods include Galerkin and least-squares methods, with the key difference being the choice of the weighting functions. This paper extends the least-squares methods to the solution of DOPs, which will allow one to obtain a number of benefits over direct collocation. Most importantly, instead of forcing the residual errors to be zero only at collocation points, we propose to minimize the integral of the residual error over the whole trajectory, allowing solutions of a much higher accuracy to be achieved for the same discretization mesh, compared to a collocation method. This new approach  also has advantages when dealing with particular types of DOPs that are challenging for direct collocation, for example those containing singular controls and high-index DAEs. 

We would like to emphasize that the method presented here is still in its early development phase. The focus of this paper is to introduce the concept of integrated residual schemes, relate them to direct collocation and to propose a new formulation, namely \emph{direct alternating integrated residual} (DAIR), as a standalone transcription method. By formally treating the process of numerically solving DOPs on a given mesh as a multi-objective optimization problem, where we trade off accuracy and optimality, the proposed alternating approach can solve a DOP  to the specified accuracy. 

Sections~\ref{sec:optimizationBasedControl}--\ref{sec:DiscretizedElementDOP} provide a brief introduction to continuous-time DOPs, their discretization and the error metrics. Following this, the conventional direct collocation method is introduced in Section~\ref{sec:DirectTranscriptionMethod} focusing on the necessity of mesh refinement. The concept of integrated residual method (IRM) is introduced in Section~\ref{sec: ResidualMinimization} with new insights on its connection to the direct collocation approach. Subsequently, in Section~\ref{sec: ProposedScheme}, the motivations for the development of the DAIR scheme is given, together with discussions on its formulation and implementation strategies. This is followed by a number of classical examples in Section~\ref{sec:ExampleProblem}, where different aspects of the method are demonstrated. In Section \ref{sec:Conclusion} we provide concluding remarks and some directions for further development. 

\section{Dynamic optimization Problem}
\label{sec:optimizationBasedControl}
A large class of optimal control, estimation and system identification problems require the solution of  dynamic optimization problems with the objective functional expressed in the general Bolza form:
\begin{subequations}
\label{eqn:DOPBolza}
\begin{equation}
\label{eqn:DOPBolzaObjective}
\min_{x,u,p,t_0,t_f} \Phi(x(t_0),t_0,x(t_f),t_f,p)
+\int_{t_0}^{t_f} L(x(t),u(t),t,p) dt
\end{equation}
subject to
\begin{align}
\dot{x}(t)=f(x(t),u(t),t,p),\ &\forall t \in [t_0,t_f]\ \text{a.e.} \label{eqn:DOPBolzaDynamics}\\
g(x(t),\dot{x}(t),u(t),t,p) = 0,\ &\forall t \in [t_0,t_f]\ \text{a.e.} \label{eqn:DOPBolzaDAEDynamics}\\
c(x(t),\dot{x}(t),u(t),t,p)\le 0,\ &\forall t \in [t_0,t_f]\ \text{a.e.} \label{eqn:DOPBolzaPathConstraint}\\
\phi(x(t_0),t_0,x(t_f),t_f,p) =0,\ &
\end{align}
\end{subequations}
 where $x: \mathbb{R} \rightarrow \mathbb{R}^n$ is the continuous \emph{state trajectory} of the system, $u: \mathbb{R} \rightarrow \mathbb{R}^m$ is the \emph{input trajectory},   $p \in \mathbb{R}^{n_p}$ are \emph{static parameters}, $t_0 \in \mathbb{R}$ and $t_f \in \mathbb{R}$ are the initial and final time.  $\Phi$ is the \emph{Mayer cost} functional ($\Phi$: $\mathbb{R}^n \times \mathbb{R} \times \mathbb{R}^n \times \mathbb{R} \times \mathbb{R}^{n_p} \to \mathbb{R}$), $L$ is the \emph{Lagrange cost} functional ($L:\mathbb{R}^n \times \mathbb{R}^m \times \mathbb{R} \times \mathbb{R}^{n_p} \to \mathbb{R}$), $f$ defines the equality constraint related to ODEs of the system ($f:\mathbb{R}^n \times \mathbb{R}^m \times \mathbb{R} \times \mathbb{R}^{n_p} \to \mathbb{R}^n$), $g$ defines the equality constraint related to the DAEs of the system ($g:\mathbb{R}^n \times \mathbb{R}^n \times \mathbb{R}^m \times \mathbb{R} \times \mathbb{R}^{n_p} \to \mathbb{R}^{n_g}$), $c$ defines the inequality \emph{path constraint} ($c:\mathbb{R}^n \times \mathbb{R}^n \times \mathbb{R}^m \times \mathbb{R} \times \mathbb{R}^{n_p} \to \mathbb{R}^{n_c}$), and $\phi$ defines the \emph{boundary condition} ($\phi:\mathbb{R}^n \times \mathbb{R} \times \mathbb{R}^n \times \mathbb{R} \times \mathbb{R}^{n_p} \to \mathbb{R}^{n_q}$). The objective functional~\eqref{eqn:DOPBolzaObjective} is often represented by a single functional~$J$ with optimal solution denoted  $J^{\ast}$, and \eqref{eqn:DOPBolzaDynamics}--\eqref{eqn:DOPBolzaDAEDynamics} are referred together as the \emph{dynamic equations} or \emph{dynamic constraints}. 
 
 \section{Discretized Dynamic Optimization Problem}
\label{sec:DiscretizedElementDOP}
Numerical discretization schemes are often used to solve dynamic optimization problems. They can be categorized into fixed-degree $h$~methods such as Euler, Trapezoidal, Hermite-Simpson~(HS) and the Runge-Kutta (RK) family \cite{betts2010practical}, and variable higher-degree $p$/$hp$~methods \cite{fahroo2008advances,patterson2014gpops}. With \emph{direct} methods, the DOP is first discretized through a transcription process, after which the resulting nonlinear programming (NLP) problem is solved numerically. Due to their simplicity in implementation, direct methods have become the de facto standard for solving practical DOPs~\cite{limebeer2015faster}, hence will be considered in this paper.

\subsection{Temporal discretization and trajectory parameterization}

First, we subdivide the domain [$t_0$, $t_f$] into $K$ intervals $\mathbb{T}_k:=[s_k,s_{k+1}]$ for  $k\in\mathbb{I}_K:=\{1,\dots,K\}$, with $s_k$ the \emph{major node} locations and $t_0=s_1 <\dots< s_{K+1}=t_f$. Inside each interval $k$, we may define additional \emph{minor nodes} depending on the requirements of different computation schemes.

The key concept employed in the transcription process is the approximation of state, state derivatives and input trajectories with parameterized continuous or piecewise-continuous functions; these will be denoted by $\tilde{x}$, $\dot{\tilde{x}}$ and $\tilde{u}$, respectively. Inside each interval $k$, the trajectory for a state variable can be approximated as
\begin{equation}
\label{eqn: LGRStateApproximation}
x^{(k)}(t) \approx \tilde{x}^{(k)}(t) := \sum_{i=1}^{N^{(k)}}a_i^{(k)}\beta_{i}^{(k)}(t),
\end{equation}
with $\beta_{i}^{(k)}(\cdot)$ a basis function and $a_i^{(k)}$ the corresponding coefficient, also known as the \emph{amplitude of the basis function} for the $i^\text{th}$ degree of freedom. If the basis functions are defined on a different domain, appropriate mapping to the time mesh must be made.

When $N^{(k)}$ Lagrange interpolating polynomials are used as basis functions, which will be the case in the remainder of this paper, the coefficients will correspond to $N^{(k)}$ points on the polynomial function $\tilde{x}^{(k)}$. In this case, these unknown coefficients are called \emph{parameterized states} and denoted by $\chi_i^{(k)}$ to distinguish this from the more general case above. We also choose to design the minor nodes corresponding to these locations on the polynomial, in which case they are called \emph{data points} $d_i^{(k)}$, i.e.\ 
$\chi_i^{(k)}=\tilde{x}^{(k)}\left(d_i^{(k)}\right) \in \mathbb{R}^{n}$ for all $i\in\mathbb{I}_{N^{(k)}}$.  Additionally, we define $\chi^{(k)}:=[\chi_1^{(k)},\dots, \chi_{N^{(K)}}^{(k)}]^{\top} \in \mathbb{R}^{N^{(k)}\times n}$ and $\chi: = [\chi^{(1)},\dots, \chi^{(K)}]^{\top} \in \mathbb{R}^{N\times n}$, with $N:=\sum_{k=1}^K N^{(k)}$, for brevity in later discussions. The parameterization of the input using the data point values $\upsilon_i^{(k)}$ can be done similarly for the approximation function $\tilde{u}^{(k)}$.

Continuity of the trajectories between interval $k$ and $k+1$ can be enforced either
\begin{itemize}
    \item implicitly by using the same decision variable for the last node of interval $k$ and the first node of interval $k+1$, reducing the number of unknowns points from $N^{(k)}$ to $N^{(k)}-1$, or
    \item explicitly by additional continuity constraints:
    \begin{equation}
    \label{eqn:continuityConstraints}
        \tilde{x}^{(k)}\left(t_{k+1}\right) = \tilde{x}^{(k+1)}\left(t_{k+1}\right).
    \end{equation}
\end{itemize}

For the inputs, they can be either implemented as continuous trajectories as for the states, or allowing them to be discontinuous at 
major nodes. In this work we will use a general formulation that considers both cases. W.l.o.g.\  we consider the grid to be given in the sense that the distribution of mesh and data points inside the domain is determined a priori; however, the time corresponding to these locations can vary, since $t_0$ and $t_f$ can be free. Hence, the solution to the discretized problem, i.e.\ the full set of optimization variables, is denoted by $\mathcal{Z} \coloneqq (\chi, \upsilon, p, t_0, t_f)$.

It is also possible to define a quadrature of different order for each interval, in order to numerically integrate a functional, e.g.\ $L$. The \emph{quadrature points} are defined as a different choice of minor nodes inside each interval $k$ according to the quadrature scheme, i.e.\ $q_i^{(k)}$, for all $i\in\mathbb{I}_{Q^{(k)}}$, where $Q^{(k)}$ is the number of quadrature points inside  interval $k$. 

With such a discretization in time and parameterization of state and input variables, we express the direct discretization of DOP~\eqref{eqn:DOPBolza} as the following NLP:
\begin{subequations}
\label{eqn:DDOPProblem}
\begin{multline}
\label{eqn:DDOPObjective}
\min_{\chi,\upsilon,p,t_0,t_f}  \Phi\left(\chi_1^{(1)},t_0,\chi_{N^{(K)}}^{(K)},t_f,p\right)\\
+\sum_{k=1}^{K}\sum_{i=1}^{Q^{(k)}} w_i^{(k)} L\left(\tilde{x}^{(k)}\left(q_i^{(k)}\right),\tilde{u}^{(k)}\left(q_i^{(k)}\right),t_0,t_f,p\right)
\end{multline}
subject to, for all $k\in\mathbb{I}_K$, 
\begin{align}
\label{eqn:DDOPEQNConstraint}
\psi^{(k)} \left(\chi^{(k)},\upsilon^{(k)},t_0,t_f,p\right) = & 0,\\
\label{eqn:DDOPIEQNConstraint}
\gamma^{(k)} \left(\chi^{(k)},\upsilon^{(k)},t_0,t_f,p\right) \le & 0,\\
\phi\left(\chi_1^{(1)},t_0,\chi_{N^{(K)}}^{(K)},t_f,p\right) =& 0.
\end{align}
\end{subequations}
as well as any necessary continuity constraints in the form of~\eqref{eqn:continuityConstraints}. The expressions for the functions $\psi^{(k)}$ and $\gamma^{(k)}$ depend on the details of the transcription method. The scalars $w_i^{(k)}, \forall i\in\mathbb{I}_{Q^{(k)}}$ are  the interval-dependent quadrature weights (i.e.\ including the corresponding time interval $\Delta t^{(k)}:=s_{k+1}-s_k$ contributions)  for the numerical integration of the Lagrange cost inside the interval.

The discretized problem can be solved with off-the-shelf NLP solvers, with discretized solution $\mathcal{Z} \coloneqq \left(\chi, \upsilon, p, t_0, t_f\right)$. In accordance to the parameterization and discretization methods employed, trajectories of the solution $\tilde{z}(t) \coloneqq \left(\tilde{x}(t), \tilde{u}(t), t, p\right)$ can be obtained. For a multiple-interval mesh, approximated state and input trajectories, $\tilde{x}$ and $\tilde{u}$, can be a piecewise polynomial based on interpolation polynomial functions $\tilde{x}^{(k)}$ and $\tilde{u}^{(k)}$ combining all intervals $k\in\mathbb{I}_K$. 

\subsection{Error metrics}
\label{sec:errormetrics}

In practice, solutions of~\eqref{eqn:DOPBolza} can rarely be represented exactly by the approximating function. For example, with a polynomial basis, even in the simple case where $f(x(t),u(t))=\dot{x}(t)=ax(t)+u(t)$ and $u(t)=1$ are both polynomials, the corresponding state trajectory $x(t)=x(0)e^{at}+\int_0^t e^{a(t-\varsigma)}u(\varsigma)\ d\varsigma$ is clearly not a polynomial for all $a\neq0$ and approximation errors should be expected. 

 To measure the accuracy of the solution, one would ideally like to compare $\tilde{x}$ with $x$, and $\dot{\tilde{x}}$ with $\dot{x}$; however such exact solutions are not obtainable for the majority of  practical problems. More importantly, many DOPs do not have a  unique optimal solution, making direct comparison of trajectories unsuitable. As a consequence, appropriate error metrics are needed that can indicate the accuracy of a solution without knowing the solution itself. 

For the dynamic equations, it is often regarded as a good idea in practice to compute the residuals $\varepsilon(t) \in \mathbb{R}^{n+n_g}$ defined as
\begin{equation}
\label{eqn: DiscretizationError}
\varepsilon(t):=\begin{bmatrix}
\dot{\tilde{x}}(t)-f(\tilde{x}(t), \tilde{u}(t), t, p)\\
g(\tilde{x}(t), \dot{\tilde{x}}(t), \tilde{u}(t), t, p)
\end{bmatrix},
\end{equation}
which is straightforward to compute for any DOP solution. However, the residual evaluated at a particular location of the domain is typically not representative of the actual accuracy of the solution. Instead, the residuals integrated along certain intervals of the domain are often found to be a suitable metric. Following this idea, the most popular choice for error analysis is to integrate the norm of $\varepsilon(t)$ for each interval $\mathbb{T}_k$ to get 
\begin{equation}
\label{eqn: AbsLocalError2}
\eta^{(k)}:=\int_{\mathbb{T}_k} 	\|\varepsilon^{(k)}(t)\|_{2}\: dt,
\end{equation} 
with $\|\cdot\|_{2}$ the vector 2-norm. The integral can be practically estimated by high-order quadrature. The metric $\eta \in \mathbb{R}^N$ is typically referred to as the \emph{absolute local error}. When this error is normalized with the largest magnitudes of state and state derivatives, the error is known as the \emph{relative local error}~\cite{betts2010practical}. Other variants are also possible such as the \emph{mean local error}, with normalization by the interval size, and \emph{squared absolute local error}, using the square of the norm instead.

Analogously, the integration of residual errors can be computed along the whole trajectory. For instance, we can compute the \emph{integrated residual norm squared} (IRNS) error as
\begin{equation}
\label{eqn: ResidualMinimizationOrgr}
r(\tilde{x}, \tilde{u}, t_0, t_f, p):= \int_{t_0}^{t_f} \|\varepsilon(t)\|^2_{2}\: dt.
\end{equation}
Other variations in the definition are also possible, e.g.\ the \emph{mean integrated residual norm squared} (MIRNS) error defined as $\frac{r}{\Delta t}$ with $\Delta t:=t_f-t_0$. Additionally, an \emph{absolute local constraint violation} $\epsilon$ may be evaluated to measure possible inequality constraint violations at different points along the trajectories. Once the errors inside the domain are evaluated, appropriate modifications can be made to the discretization mesh. The problem can be solved iteratively until a solution that fulfills all predefined error tolerances is obtained. This process is commonly known as \emph{mesh refinement} \cite{betts2010practical}.  

\subsection{Enforcement of dynamic constraints in the NLP}
\label{subsec:ErrorinFunctionApproximation}

The inevitability of approximation errors leads to an important implication: it is not possible for \eqref{eqn:DOPBolzaDynamics}--\eqref{eqn:DOPBolzaDAEDynamics} to be satisfied everywhere along the domain  for any arbitrary choice of the minor nodes. 
To gain better insight on how these constraints should be dealt with, we refer to a broader class of numerical methods commonly used to solve differential equations. One way is to define an equivalent optimization problem, which minimizes a measure of the solution, based on the error criteria. This is known as the Rayleigh-Ritz approach~\cite[Sect.~5.2--5.7]{rao2010finite} and the optimization problem can either be solved directly or through the use of some optimality conditions. A related, but more generally applicable approach than the Rayleigh-Ritz method is the method of weighted residuals~\cite[Sect.~5.8]{rao2010finite}, requiring 
\begin{equation}
\label{eqn: WRMWeakForm}
    \int_{\mathbb{T}_k} \varpi^{(k)}(t)\varepsilon_{j}^{(k)}(t)\: dt =0, \text{ for } j=1,\hdots,n+n_g,
\end{equation}
for all weighting functions  $\varpi^{(k)}:\mathbb{R} \rightarrow \mathbb{R}$ taken from a suitably-defined set of functions. The use of such weighting functions essentially provides a way to test the value of the local residuals. Thus, $\varpi^{(k)}$ is also commonly referred to as a \emph{test function} or \emph{trial function} in the literature. 

When yielding finite-dimensional approximations, we choose a finite set of weighting functions as test functions. Different choices of test functions lead to different variants of weighted residual methods, such as Galerkin, collocation, least-squares and the method of moments. They each have their own properties for solution accuracy and computational complexity and the appropriate choices will be problem-dependent. Regardless of the choice, the resultant set of equations can be implemented as the equality constraints \eqref{eqn:DDOPEQNConstraint} so that the dynamic equations \eqref{eqn:DOPBolzaDynamics}--\eqref{eqn:DOPBolzaDAEDynamics} can be approximately satisfied.

\section{Direct collocation}
\label{sec:DirectTranscriptionMethod}
For the collocation weighted residual method, the test functions are selected to be Dirac delta functions, leading to $n+n_g$ equality constraints to be applied to each of the $N^{(k)}$ data points. The Dirac delta functions posses the \emph{isolation property}, namely that the integral of the function on an interval is zero,  except the intervals that 
contain the center of the function, where the integral equals to 1. Therefore information needed to evaluate a constraint equation at a data point will be fully independent from information corresponding to other data points, contributing to the computational efficiency of the direct collocation method. 

The other simplification commonly made in direct collocation is to also use the same data point definition 
for both the quadrature points in the numerical integration of the Lagrange cost and the points where path constraints \eqref{eqn:DOPBolzaPathConstraint} are forced to be satisfied. As a result, the major nodes and data points together would be sufficient for the transcription of the problem to an NLP, and the data points in this case are known as the \emph{collocation points}.

With the collocation weighted residual method, the resultant equality constraints from~\eqref{eqn: WRMWeakForm} with a finite-dimensional approximation is
\begin{subequations}
\label{eqn:DCConstraints}
\begin{align}
\label{eqn:DCODEDynamicsConstraint}
\sum_{l=1}^{N^{(k)}}\mathcal{A}_{il}^{(k)}\chi_l^{(k)}+\mathcal{D}_{il}^{(k)}f\left(\chi_l^{(k)},\upsilon_l^{(k)},t_0,t_f,p\right) = & 0, \\
\label{eqn:DCDAEDynamicsConstraint}
g\left(\chi_i^{(k)},\dot{\chi}_i^{(k)},\upsilon_i^{(k)},t_0,t_f,p\right) = & 0, \end{align}
with $\dot{\chi}_i^{(k)}:=\dot{\tilde{x}}^{(k)}\left(d_i^{(k)}\right) \in \mathbb{R}^{n}$. $\mathcal{A}^{(k)} \in \mathbb{R}^{N^{(k)}\times N^{(k)}}$ is a discretization-dependent constant matrix, where $\mathcal{A}_{il}^{(k)}$ is element $(i,l)$ of the matrix, and $\mathcal{D}^{(k)} \in \mathbb{R}^{N^{(k)}\times N^{(k)}}$ is a matrix containing time variables. In NLPs arising from direct collocation, \eqref{eqn:DDOPEQNConstraint} needs to contain \eqref{eqn:DCODEDynamicsConstraint}--\eqref{eqn:DCDAEDynamicsConstraint} for the dynamic equations to be approximately fulfilled. Also, \eqref{eqn:DDOPIEQNConstraint} is chosen such that the inequality constraints are
\begin{equation}
\label{eqn:DCPathConstraint}
c\left(\chi_i^{(k)},\dot{\chi}_i^{(k)},\upsilon_i^{(k)},t_0,t_f,p\right) \le 0.
\end{equation}
\end{subequations}

In essence, direct collocation forces the residuals $\varepsilon^{(k)}$ to be zero at all collocation points. It is well-known in the field of approximation theory that if a function cannot be represented exactly by a polynomial, forcing the approximating polynomial to exactly go through some sampled data points generally results in larger errors for the function values between the data points than other methods, such as least-squares fitting. Similarly, direct collocation will generally result in large errors between collocation points, regardless of the distribution and spacing of these points.

More importantly, since most direct collocation methods employ absolute or relative local error as the error metric, e.g.\ with \eqref{eqn: AbsLocalError2}, a mismatch arises between the error measures in the problem formulation and the error criteria for a solution to be sufficiently accurate:
\begin{itemize}
    \item when formulating the NLP, satisfaction of dynamic constraints are based on residuals \emph{at collocation points}, whereas
    \item during the error analysis of the solution, satisfaction of dynamic constraints are based on the norm of residuals \emph{integrated along intervals in-between collocation points}.
\end{itemize}
As a direct consequence and often not realized by non-experts, 
\begin{itemize}
    \item regardless of how small the solver tolerances are, solving the NLP once on a single given discretization mesh will provide no guarantee in terms of solution accuracy and constraint satisfaction, hence
    \item posterior procedures such as error analysis and mesh refinement must be considered as an indispensable part of a direct collocation method to ensure convergence and solution accuracy.
\end{itemize}
Integrated residual methods fundamentally address the problems arising from this error metric mismatch, by working with the residuals in integrated form in the NLP formulations.

\section{Integrated residual methods}
\label{sec: ResidualMinimization}

In the field of approximation theory, the least squares criterion is often considered as a more suitable choice than forcing the fitting error to be exactly zero only at some selected points~\cite{BirgeRaymondLeastSquares}. Before exploring the implementation of the least-squares approach for the solution of DOPs, we first look at the use of such a method in solving dynamic equations in the form of ODEs and DAEs. For the reminder of this work, we will use the MIRNS error as the error metric, however other variants of IRNS error may be selected also.  

  Following the Rayleigh-Ritz approach, we can equate finding an approximate solution of the dynamic equations to the following optimization problem that minimizes the MIRNS error:

\begin{equation}
\label{eqn: ResidualMinimizationOrg}
\min_{\chi,\upsilon,p,t_0,t_f}  \frac{1}{\Delta t}r(\tilde{x}, \tilde{u}, t_0, t_f, p)
\end{equation}
subject to any continuity and boundary constraints.

The least squares approach as defined in the class of weighted residual methods is equivalent to applying the optimality conditions and obtaining a number of equality constraints to be satisfied. However, we note that this condition is only necessary and thus theoretically can only guarantee that the trajectory is a stationary solution in general, i.e.\ the trajectory could be a local maximum for the MIRNS error. In addition, using only the optimality conditions will not be able to provide indications on the magnitudes of the errors. Hence, in this work we will focus on the development of methods that directly solve the optimization problem \eqref{eqn: ResidualMinimizationOrg} instead, with the added benefit that the evaluation of error magnitudes can be integrated into the solution process instead of a posteriori. 

The cost formulation~\eqref{eqn: ResidualMinimizationOrgr} effectively introduces relative trade-offs for the accuracy between the dynamic equations. Although the original expression works well when all variables are scaled to the same numerical range,  there exist situations where it may be beneficial to specify additional weighting terms with a diagonal matrix $\mathcal{W} \in \mathbb{R}^{(n+n_g)\times(n+n_g)}$ for the corresponding dynamic equations. In practice, we often know beforehand that the modeling of some relationship (e.g.\ between acceleration and velocity) will have a higher confidence level than the modeling of some other dynamics (e.g.\ relationship between gas peddle position and acceleration). In these cases, it is preferable to formally specify what would be the desired trade-off in terms of accuracy for different dynamic equations.

For certain simple problems, $r$ may be expressed analytically for precise computation. However, for the majority of practical problems, numerical integration with quadrature rules of sufficiently high order can be used, i.e. the objective in \eqref{eqn: ResidualMinimizationOrg} can be replace by
\begin{subequations}
\label{eqn:LS}
\begin{equation}
\label{eqn: ResidualMinimizationDiscrete}
\min_{\chi,\upsilon,p,t_0,t_f} \frac{1}{\Delta t}
\sum_{k=1}^{K} \mathcal{R}\left(\chi^{(k)},\upsilon^{(k)},t_0,t_f,p\right)
\end{equation}
with
\begin{equation}
\label{eqn: ResidualMinimizationDiscreteQuadrature}
\mathcal{R}\left(\chi^{(k)},\upsilon^{(k)},t_0,t_f,p\right):=\sum_{i=1}^{Q^{(k)}}w_{i}^{(k)}\begin{Vmatrix}\mathcal{W}\varepsilon(q_i^{(k)})\end{Vmatrix}^2_{2}.
\end{equation}
\end{subequations}

\subsection{Least squares method for solving the DOP}

When solving the DOP using the least squares approach, instead of just solving the differential equations, it is important to address the relationship between the requirement to minimize the integrated residual~\eqref{eqn: ResidualMinimizationDiscrete} and the desire to minimize the original objective~\eqref{eqn:DOPBolzaObjective}. With an indirect approach, the optimality conditions for the DOP can be formulated and subsequently solved using least-squares finite element methods~\cite{bochev2006least}. 

For direct transcription methods, recently proposed penalty-barrier finite element method (PBF)~\cite{9304216} formulates an augmented objective consisting of the original objective, the MIRNS error as a penalty term and inequality constraint violations as an integrated logarithmic barrier term. The authors were able to prove  convergence of their method provided that the functions that define the problem satisfies appropriate boundedness and Lipschitz conditions. 

Based on the same concept of minimizing the integrated residual, our earlier work~\cite{NieSolutionRep} presented a solution representation method that is able to obtain solutions of much higher accuracy than collocation methods, while maintaining non-increasing objective values. This approach effectively treats computing an approximate solution of the DOP as a multi-objective optimization problem. In this work, we extend this method as a stand-alone scheme for the solution of DOPs and the new method will be addressed in detail in Section~\ref{sec: ProposedScheme}. Here, we will first focus on demonstrating some of the characteristics of this class of methods, and the relationship to  direct collocation. 

\subsection{Relationship to direct collocation}
\label{subsec: relationshiptoDC}
In general, IRM is considered as a different approach to collocation. Here, we show a different perspective, namely that direct collocation can be considered as a special case of IRM.

\begin{proposition}
\label{prop: RelationshipIRMDC}
In the direct collocation formulation, the enforcement of dynamic constraints with~\eqref{eqn:DCODEDynamicsConstraint}--\eqref{eqn:DCDAEDynamicsConstraint} is equivalent to the solution of a special case of problem \eqref{eqn:LS}, with the quadrature points $q_i^{(k)}, \forall i\in\mathbb{I}_{Q^{(k)}}$ in each interval selected to be the same as the data points $d_i^{(k)}, \forall i\in\mathbb{I}_{N^{(k)}}$ of that interval. 
\end{proposition}
\begin{proof}
For equality constraints~\eqref{eqn:DCODEDynamicsConstraint}--\eqref{eqn:DCDAEDynamicsConstraint} to be fulfilled, the residuals $\varepsilon(t)$ evaluated at collocation points $d_i^{(k)}$ for all $i\in\mathbb{I}_{N^{(k)}}$ and $k\in\mathbb{I}_{K}$ will all be zero. When the quadrature points chosen for the IRM problem~\eqref{eqn:LS} match the data points,  the IRM problem will have the optimal solution with $\mathcal{R}^{\ast}=0$,  since the corresponding residuals $\varepsilon(t)$ at these data points can be forced to zero altogether. This is equivalent to enforcing \eqref{eqn:DCODEDynamicsConstraint}--\eqref{eqn:DCDAEDynamicsConstraint}.
\end{proof}

Hence, the enforcement of dynamic constraints in direct collocation can be interpreted as an IRM where in~\eqref{eqn:LS} the quadrature points are chosen to be the same as the data points. The quadrature order with this choice is not sufficiently high, in general. Hence, large errors may occur between the collocation/polynomial data points, which will not be reflected in any convergence and error measure of the underlying NLP. Hence, successfully solving the direct collocation NLP to very small tolerances does not guarantee an accurate solution. 

\subsection{Improved DAE handling}
For direct collocation, if DAE equations exist as part of the dynamics, in addition to fulfilling all other constraints, there may not always be sufficient remaining degrees of freedom to additionally satisfy equation~\eqref{eqn:DCDAEDynamicsConstraint} for all $i\in\mathbb{I}_{N^{(k)}}$, causing convergence issues for the NLP solver. The opposite could happen as well, with degree of freedoms not uniquely defined by the constraints, leading to multiple or even an infinite number of solutions. In this case, significant fluctuations will occur in the obtained solution. If the original continuous-time DOP is consistent, inconsistencies as described above would be attributed to the constraint discretization process, leading to either an \emph{over-constrained} or \emph{under-constrained} NLP. 

Without special considerations as discussed in \cite{logsdon1989accurate}, direct collocation methods are known to struggle for high-index DAE systems \cite{betts2010practical}, and systems with constraints that force the solution to lie on a manifold. For instance, in three-dimensional mechanical systems with quaternions: in addition to implicitly determined forces, a quaternion equation constrains the solutions to lie on a unit sphere. With these types of problems, convergence of the NLP solver may be significantly deteriorated if a good initial guess is not provided. By allowing arbitrarily small residuals for the constraints to exist during the solution process similar to penalty methods \cite{fiacco1990nonlinear}, IRMs have been demonstrated to have better convergence properties in both cases of high-index DAEs and DAEs that force the solution to lie on a manifold.

\subsection{Suppression of singular arc fluctuations}
\label{subsec:SuppressionSARC}

As explained in Section~\ref{sec:errormetrics}, in the parameterisation of DOPs, the representation of the state and input trajectories can rarely be made exact, hence approximation errors are generally unavoidable. This provides a unique opportunity for IRM-type transcription methods to automatically suppress potential singular arc fluctuations, without the need for additional treatments.

Due to the existence of approximation errors, and the multi-objective nature of IRM for solution of DOPs, different solution candidates on the singular arc with negligible differences from the objective point of view can now be ranked by the error. Larger fluctuations in the solution generally lead to bigger errors along the trajectory, therefore a solution with the smallest fluctuations is often the most accurate solution in the IRM residual error metrics. This is the key reason behind the suppression of singular arc fluctuation with IRM-type transcriptions.

\section{Direct Alternating Integrated Residual (DAIR) Method}
\label{sec: ProposedScheme}

Though  PBF  has a number of advantages over direct collocation in terms of solution accuracy and robust handling of some difficult problems, PBF is still a method developed focusing on off-line solution of dynamic optimization problems. Illustrated in Figure~\ref{fig:PF_TrajOpt}, these types of problems often have one target solution that the DOP algorithm is searching for, namely the solution with the smallest objective value among the ones that contain the lowest possible error. In other words, one wants the convergence of the objective ($J \to J^{\ast}$) and constraint satisfaction ($r,\eta,\epsilon \to 0$) at the same time, as the discretization mesh becomes denser ($K$ increases). 

\begin{figure}[t]
    \centering
    \begin{subfigure}[b]{\columnwidth}
        \centering
        \includegraphics[width=\textwidth]{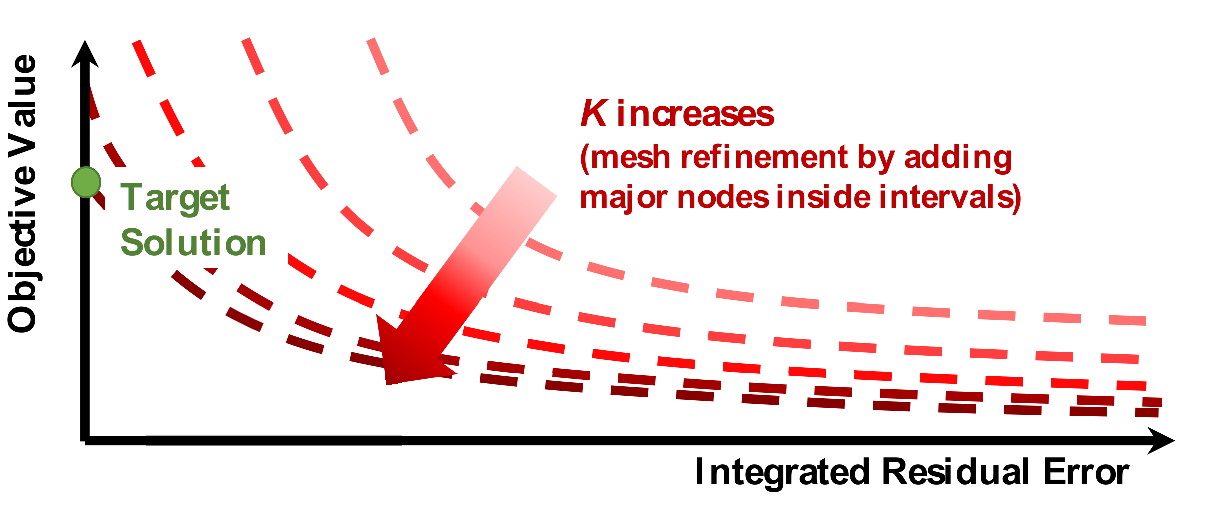}
    	\caption{Dynamic optimization}
    	\label{fig:PF_TrajOpt}
    \end{subfigure}
	\begin{subfigure}[b]{\columnwidth}
        \centering
        \includegraphics[width=\textwidth]{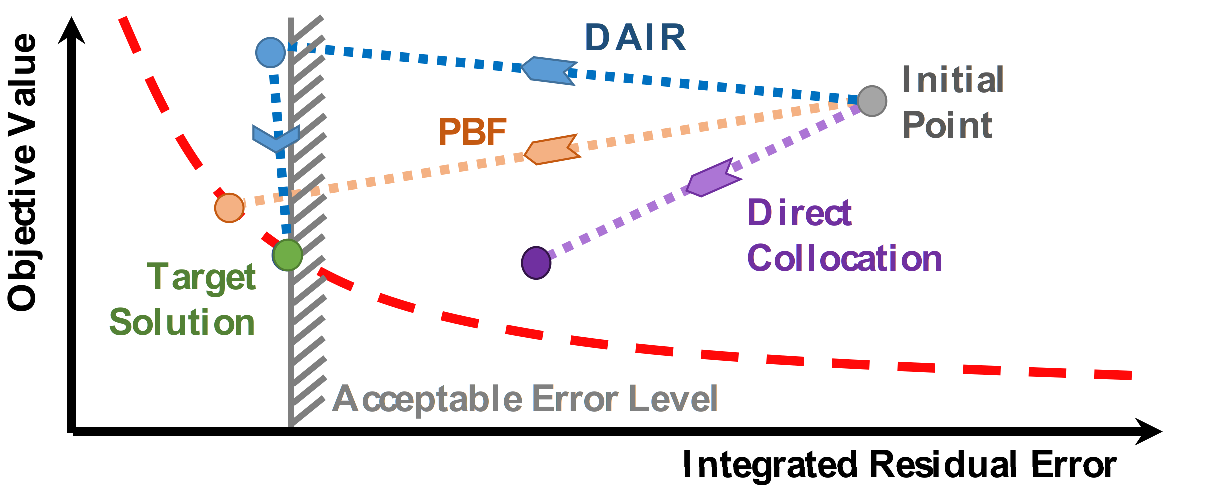}
    	\caption{Single solve on a given mesh}
    	\label{fig:PF_FixedMesh}
    \end{subfigure}
    \caption{Illustration for the differences in solving dynamic optimization problems offline and solving DOPs on a single given mesh.}
\label{fig:ParetoFrontComp}
\end{figure}

The picture is different when considering a single solve of the DOP on a given discretization mesh, especially considering NMPC applications. Firstly, the nature of solving the DOP numerically will become a multi-objective problem, leading to an inevitable trade-off between minimizing the objective and reducing the residual error. This often indicates that, in practice, the target solutions will not be the ones that lie on the far ends of the Pareto front.

Evaluating preferences among various solutions on the Pareto front depends on other criteria, e.g.\ the closed-loop performance of the NMPC controller. We refer to other work, e.g.\ \cite{BulatParetoFront}, for details on how such a trade-off can be made. Here we directly take the outcome of this decision-making process: an error level under which the solution accuracy can be considered acceptable. The original multi-objective optimization problem can then be translated into a single objective one, with the target solution being the one that minimizes the objective value, while satisfying the constraints concerning the acceptable error level. 

In Figure \ref{fig:PF_FixedMesh}, we illustrate the solution process of direct collocation and the PBF method for a given mesh size and discretization method. As long as the initial point and mesh design do not change, the solution that a direct collocation method can obtain will not change. From earlier discussions, it can be seen that, regardless of whether this solution satisfies the acceptable error level, it is very unlikely in practice to be a solution that resides on the Pareto front. In other words, one aspect of a direct collocation solution can be improved without deteriorating the other aspect.

In contrast, PBF is capable of finding solutions on the Pareto front; however, controlling which solution it will  terminate at would require careful selection of parameters for the penalty and barrier terms. Therefore, even if the sub-iterations can be computed efficiently, additional challenges are associated with the PBF method to converge easily to the target solution, given a specified acceptable accuracy level. The proposed DAIR method aims to address these challenges and provide a reliable and efficient approach.

\subsection{Elementary formulations}
The elementary formulations of the DAIR method consist of two problems: minimizing the MIRNS error  and minimizing the objective subject to integrated residual error constraints, denoted as the \emph{DAIR residual minimization problem} and \emph{DAIR cost minimization problem}, respectively. The method retains the same decision variables as in \eqref{eqn:DDOPProblem}, namely $\mathcal{Z} \coloneqq (\chi, \upsilon, p, t_0, t_f)$, and uses the interpolation polynomial formula $\tilde{x}(\cdot)$, $\dot{\tilde{x}}(\cdot)$ and $\tilde{u}(\cdot)$  for the computation and integration of various elements of the discretized DOP. The interpolation formulation is provided in more detail in our previous work on solution representation methods~\cite{NieSolutionRep}.

The \emph{DAIR residual minimization problem} has the  formulation
\begin{subequations}
\label{eqn: AlternatingResMin}
\begin{equation}
\label{eqn: AlternatingResMinCost}
\min_{\chi,\upsilon,p,t_0,t_f}
\frac{1}{\Delta t}
\sum_{k=1}^{K} \mathcal{R}\left(\chi^{(k)},\upsilon^{(k)},t_0,t_f,p\right)
\end{equation}
subject to, for all $i\in\mathbb{I}_{N^{(k)}}$ and $k\in\mathbb{I}_{K}$,
\begin{align}
\label{eqn: AlternatingResMinPathConstraint}
c\left(\chi_i^{(k)},\dot{\chi}_i^{(k)},\upsilon_i^{(k)},t_0,t_f,p\right) \le & 0,  \\
\label{eqn: AlternatingResMinBoundaryConstraint}
\phi\left(\chi_1^{(1)},t_0,\chi_{N^{(K)}}^{(K)},t_f,p\right) =& 0,
\end{align}
and optionally any continuity constraints in the form of~\eqref{eqn:continuityConstraints}, as well as optionally one or more constraints  from the following constraints regarding upper limits for the objective $J_c \in \mathbb{R}$:
\begin{multline}
\sum_{k=1}^{K}\sum_{i=1}^{Q^{(k)}} w_i^{(k)} L\left(\tilde{x}^{(k)}\left(q_i^{(k)}\right),\tilde{u}^{(k)}\left(q_i^{(k)}\right),t_0,t_f,p\right) \\
+\Phi\left(\chi_1^{(1)},t_0,\chi_{N^{(K)}}^{(K)},t_f,p\right) \le J_c,
\end{multline}
and the mean integrated residual squared (MIRS) error for individual dynamic equations $\varrho \in \mathbb{R}^{(n+n_g)}_{\geq 0}$:
\begin{equation}
\label{eqn: AlternatingResResConstraint}
\sum_{k=1}^{K}\sum_{i=1}^{Q^{(k)}}\frac{w_{i}^{(k)}}{\Delta t}(\varepsilon_{j}^{(k)}(q_i^{(k)}))^2\le \varrho_{j}, \text{ for } j=1,\hdots,n+n_g,
\end{equation}
\end{subequations}
with $\varrho_{j}$ the $j^\text{th}$ element in $\varrho$.

The counterpart, the \emph{DAIR cost minimization problem}, is
\begin{multline}
\label{eqn: AlternatingCostMin}
\min_{\chi,\upsilon,p,t_0,t_f}  \Phi(\chi_1^{(1)},t_0,\chi_{N^{(K)}}^{(K)},t_f,p)+\\
\sum_{k=1}^{K}\sum_{i=1}^{Q^{(k)}} w_i^{(k)} L\left(\tilde{x}^{(k)}\left(q_i^{(k)}\right),\tilde{u}^{(k)}\left(q_i^{(k)}\right),t_0,t_f,p\right)
\end{multline}
subject to~\eqref{eqn: AlternatingResMinPathConstraint}, \eqref{eqn: AlternatingResMinBoundaryConstraint}, \eqref{eqn: AlternatingResResConstraint} and~\eqref{eqn:continuityConstraints}, for $i\in\mathbb{I}_{N^{(k)}}$ and $k\in\mathbb{I}_{K}$.

In terms of the accuracy of the solution, the  above two problems can be considered as a practically balanced approach: DAIR is more reliable than direct collocation because it minimizes the MIRNS error for the dynamic equations. DAIR is easier to implement and solve than the PBF method, because DAIR avoids the need to introduce a sequence of weights for penalty and barrier terms and choosing an appropriate, tailored NLP solver.  The NLPs in DAIR will also benefit from tailored solvers designed for efficiency, but the  NLPs can be successfuly solved directly using most off-the-shelf solvers. 

Note that inequality constraints are chosen to be enforced at polynomial data points only (similar to direct collocation), and existing constraint tightening techniques (e.g.\ in \cite{Paiva8392794}) may be applied when necessary. This is an efficient choice for numerical computations, but is without loss of generality; the DAIR framework allows discretized inequality constraints to be enforced anywhere along the trajectory. 

\subsection{Implementation strategies}

\subsubsection{Standalone direct transcription method}
Based on the elementary formulations, the DAIR framework can be implemented as a standalone method for solving the DOP numerically on a given discretization mesh, with one example illustrated in Figure~\ref{fig:DAIRFlowChart}. 
\begin{figure}
\centering
\scalebox{0.75}{\begin{tikzpicture}[align=center, scale=1, node distance = 2.2cm and 1cm]
    \node [block] (meshDesign) {Design/select the discretization mesh, determine the required accuracy level in the form of MIRS errors};
    \node [blockc1, below of=meshDesign] (minResSolve) {Solve the DAIR residual minimization problem with early termination};
    \node [decision1, below of=minResSolve, aspect=2] (errorCheck) {MIRS error for each dynamic equation within requirement?};
    \node [blockc3, below of=errorCheck] (selectIR_R) {Select MIRS error upper bounds $\varrho$ based on pre-set requirement};
    \node [blockl4, left of=errorCheck] (selectIR_A) {Select MIRS error upper bounds $\varrho$ based on achievable error levels};
    \node [blockc2, below of=selectIR_R] (minCostSolve) {Solve the DAIR cost minimization problem};
    \node [blockr1, right of=minCostSolve] (stop) {Stop};
    
    \path [line] (meshDesign) -- (minResSolve);
    \path [line] (minResSolve) -- (errorCheck);
	\path [line] (errorCheck) -- node [color=black,above] {no}(selectIR_A);
    \path [line] (errorCheck) -- node [color=black,right] {yes}(selectIR_R);
    \path [line] (selectIR_R) -- (minCostSolve);
    \path [line] (minCostSolve) -- (stop);
    \path [line] (selectIR_A) |- (minCostSolve);

\end{tikzpicture}}
\caption{Overview of DAIR scheme as a standalone method for solving DOP numerically on a given discretization mesh}
\label{fig:DAIRFlowChart}
\end{figure}
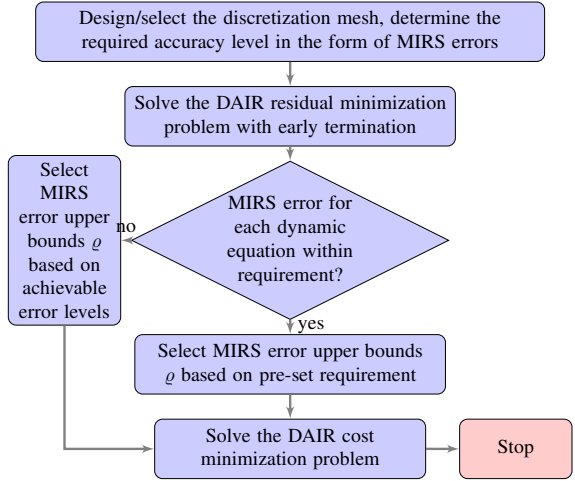

The first step is to select a discretization method, design the mesh and (optionally) determine the weighting parameter for the residual norm computation. Also, the required accuracy level for each dynamic equation needs to be specified in the form of a MIRS error. The idea is to first solve the DAIR residual minimization problem to determine the MIRS error upper bound for the DAIR cost minimization problem formulation.

For the DAIR residual minimization, a set of criteria can be specified in the NLP solver to terminate early once all MIRS errors are within requirement and all other constraints satisfied. This indicates the existence of solutions for this discretization mesh that would fulfill all accuracy requirements, and the DAIR cost minimization problem can be solved subsequently with $\varrho$ configured accordingly. 

On the other hand, if the required MIRS errors are not achievable for this mesh design, early termination will not be triggered and the DAIR residual minimization problem will be fully solved. From the solution, one can extract the smallest MIRS error achievable for any corresponding dynamic equation for which the original requirement cannot be met, and implement this achievable MIRS error or a relaxed version of it as corresponding constraint bounds, to ensure the existence of feasible solutions for the DAIR cost minimization problem. In other words, the DAIR cost minimization problem is guaranteed to have at least one feasible solution, namely the solution at which the DAIR residual minimization problem terminates.

\subsubsection{Solution representation method}
In our early work~\cite{NieSolutionRep}, we proposed an optimization formulation for representing continuous DOP trajectories with higher accuracy from discretized direct collocation NLP solutions. The DAIR residual minimization problem would be a more suitable candidate for the purpose of solution representation, with the following benefits:
\begin{itemize}
    \item there is a flexible trade-off between the level of accuracy for different dynamic equations.
    \item DAIR allows an upper limit for both the objective and MIRS errors for individual dynamics to be set. This guarantees that the obtained trajectory will be no worse than the collocation solution in terms of optimality and accuracy. 
\end{itemize}
All results shown in Section \ref{sec:ExampleProblem} use the DAIR formulation. 

\subsubsection{Other implementation potentials}
The flexibility of the DAIR scheme could enable the formulation of various implementation procedures that are based on it, for a wide range of applications. For example, the scheme can be designed for efficient and accurate solution of dynamic optimization problems, both on-line and off-line, when used together with a suitable mesh refinement/adaptation scheme. This is especially beneficial for on-line NMPC, where the solution accuracy cannot be ensured with a single mesh that has been designed off-line.

\subsection{Order of convergence for the dynamic equations}
 Due to the fact that many DOPs do not have a unique solution, and the multi-objective nature of solving DOPs numerically with discretization methods, the order of convergence for the solution of DOPs using IRM is a sophisticated topic and beyond the focus of this paper. Nevertheless, we would like to note that 
the order of convergence, measured based on the order at which the error reduces with decreasing interval sizes, has been observed in practice to be similar to the order of convergence for solving ODEs using the least squares method~\cite{ascher1978}.  This order of convergence is the same as other DOP transcription methods, such as direct collocation, because the error is eventually dominated by the degree of the finite element polynomial as the interval size tends to zero~\cite{hairer2010solving}. This behavior is also demonstrated in Example~\ref{subsec: ExampleCarPole} with Figure~\ref{fig:CartPole_Sol_Comp_HS8_ResError}.

\section{Example Problems}
\label{sec:ExampleProblem}
To demonstrate the advantages of the DAIR method over direct collocation, three example problems are presented to focus on different aspects. All problems are transcribed using the toolbox \texttt{ICLOCS2} \cite{ICLOCS2}, and solved with interior point NLP solver \texttt{IPOPT} \cite{wachter2006implementation} to a relative convergence tolerance (\texttt{tol}) of $10^{-9}$.  The current version of \texttt{ICLOCS2} has an experimental implementation of the DAIR transcription method for the purpose of proof-of-concept.  

\subsection{Goddard rocket}
As the first example, we will demonstrate different implementations of the Goddard rocket problem \cite[Ex.\ 4.9]{betts2010practical}. The optimal solution is in the form of \emph{bang-singular-bang} and, on the singular arc, it is known for the solution to be oscillatory when solved directly with a single phase numerical solver. Such fluctuations can be clearly seen in Figure~\ref{fig:GoddardRocket_SolComp}. The conventional way of dealing with singular control problems is to introduce additional conditions once the solution structure is known. Despite yielding an accurate solution, this method, however, would normally require a multi-phase formulation support and analytical derivations of the \emph{singular arc conditions}\cite{betts2010practical}.

\begin{figure}[t]
\begin{center}
\includegraphics[width=0.9\columnwidth]{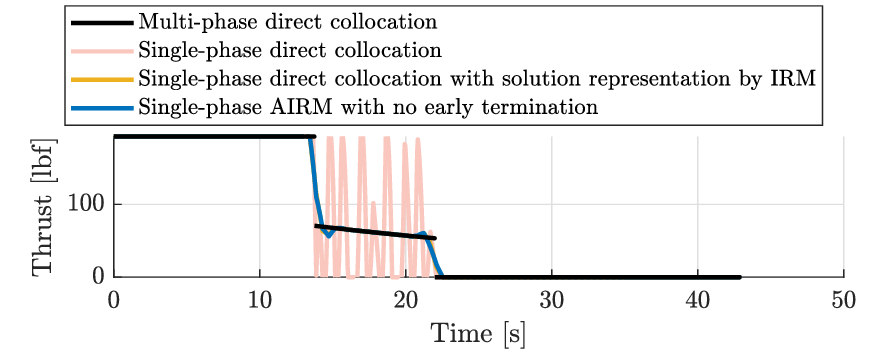}    
\caption{Solutions for the Goddard rocket problem (HS/piecewise cubic parameterization with 100 major nodes)} 
\label{fig:GoddardRocket_SolComp}
\end{center}
\end{figure}

By taking care of the errors in solution trajectories between polynomial data points, both the proposed DAIR scheme and IRM solution representation method derived from the DAIR residual minimization problem are capable of reproducing this multi-phase solution using the original single phase formulation. This is illustrated in Figure~\ref{fig:GoddardRocket_SolComp} with two small oscillations due to approximating the discontinuous optimal input trajectory with a continuous trajectory, whereas the multi-phase setup allows for a discontinuous input. 

\subsection{High-index DAE system}
To demonstrate the advantages of DAIR in dealing with high-index DAE systems, we use the example from \cite[equation system 3]{campbell2016solving} with a DOP derived from a pendulum system containing a index-3 DAE. The authors \cite{campbell2016solving}  found that existing direct collocation solvers,  such as GPOPS-II \cite{patterson2014gpops}, all failed to solve the problem directly in this formulation. For these solvers to yield a solution, problem reformulation and DAE index reduction procedures are necessary. 

We found the same behaviour with the direct collocation implementation in ICLOCS2, shown in Figure~\ref{fig:HighIndexDAE_Comp}. However, the DAIR method is able to solve the problem directly without difficulties, by minimizing the integrated residuals of the DAE system, instead of forcing the residuals to be zero at collocation points. 

\begin{figure}[t]
\begin{center}
\includegraphics[width=0.9\columnwidth]{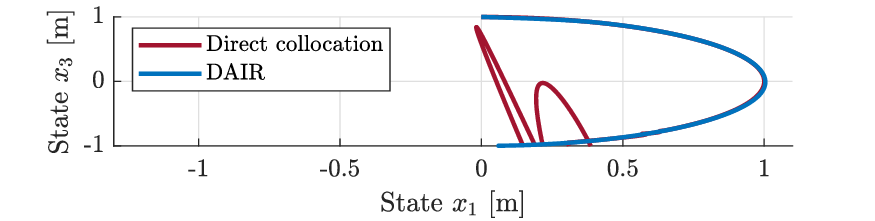}    
\caption{Solutions for high-index DAE system (LGR discretization with 8 mesh intervals, each interval with polynomial degree 5)} 
\label{fig:HighIndexDAE_Comp}
\end{center}
\end{figure}

\subsection{Cart pole swing-up}
\label{subsec: ExampleCarPole}
The cart pole swing-up problem from \cite{kelly2017introduction} requires movement of the cart to a specific location while making sure the pendulum attached to it achieves a vertically-up orientation at $t_f=2$\,s. The problem has  the position of the cart $y_1$ and  the angle of the pendulum arm $\theta_1$; these are state variables together with their time derivatives $\dot{y}_1$ and $\dot{\theta}_1$. The control input is $u \in [-20, 20]$ (force in Newtons). The following terminal conditions are imposed: $y_1(t_f)= 1$\,m, $\dot{y}_1(t_f)=0$\,m/s, $\theta_1(t_f) = \pi$\,rad, and $\dot{\theta}_1(t_f)=0$\,rad/s.

Figure~\ref{fig:CartPole_Sol_Comp_HS8} illustrates the comparison between the direct collocation and DAIR solutions, solved on the same given and coarse discretization mesh. Although the NLP problem transcribed via the direct collocation method successfully terminated with negligibly small tolerances, it becomes apparent that if the corresponding input trajectory is to be applied, the actual evolution of the system states will be very different than what was predicted by the DOP solution. Subsequently, at the final time, the state variables are far away from the terminal conditions. After examining the absolute local error $\eta$ for each mesh interval, the discrepancies in the solution trajectories can be attributed to the large residual errors arising from trajectories between collocation points. 
In constrast, the DAIR method without any early termination criteria yields a solution of very high accuracy considering the very coarse mesh employed. This, however, comes at a cost with a much higher objective value (indication of control effort), further emphasising the multi-objective nature of solving a dynamic optimization problem on a single discretization grid.

\begin{figure}[tb]
\begin{center}
\includegraphics[width=\columnwidth]{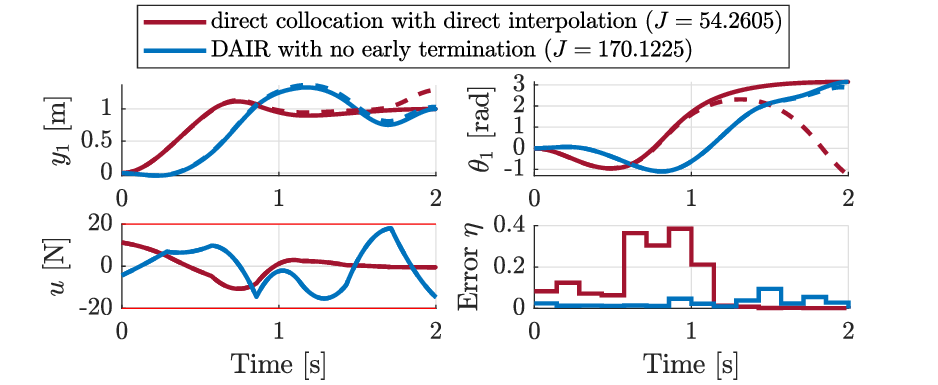}    
\caption{ Solutions to the cart pole swing-up problem (HS/piecewise cubic with 7 mesh intervals, solid lines represent trajectories as solver output, dashed lines represent the resultant trajectory by implementing the input trajectory) } 
\label{fig:CartPole_Sol_Comp_HS8}
\end{center}
\end{figure}

For further exploration, Figure~\ref{fig:CartPole_Sol_Comp_HS8_PF} highlights the trade-off between solution accuracy and optimality. The figure shows a distinctive Pareto front formed by multiple DAIR solutions with different termination conditions depending on the requested error magnitudes. Direct collocation, on the other hand, is only capable of generating a single solution that is  clearly dominated by DAIR solutions. This demonstrates the advantage of the DAIR scheme over direct collocation in terms of flexibility and Pareto optimality. 

\begin{figure}[tb]
\begin{center}
\includegraphics[width=\columnwidth]{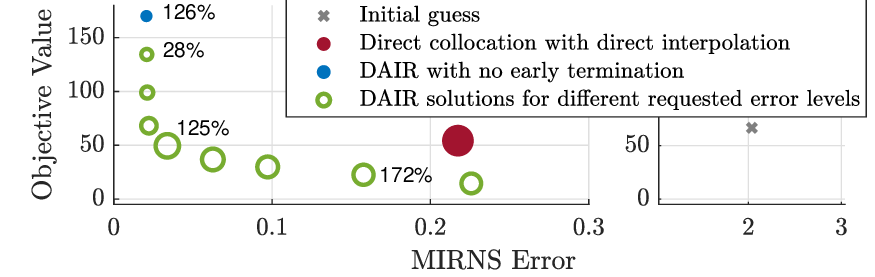}    
\caption{Trade-off between solution accuracy and optimality for the cart pole problem (HS/piecewise cubic with 7 mesh intervals, size of a circle is proportional to the constraint violation at~$t=t_f$. Percentage tags represent the increase in computation time compared to direct collocation solutions with the same error, i.e.\ with a suitably higher number of mesh intervals through mesh refinement.)} 
\label{fig:CartPole_Sol_Comp_HS8_PF}
\end{center}
\end{figure}

It is also important to note that due the system being open-loop unstable, a closed-loop implementation of the DOP solution will be necessary in practice. When implemented on-line as in NMPC, there will be another trade-off process between the DOP solution accuracy and closed-loop performance --- the most accurate open-loop DOP solution may not always be preferred~\cite{BulatParetoFront}. Therefore the flexibility of the DAIR to reliably solve a DOP to a specific accuracy level while ensuring Pareto optimality makes it a highly desirable method.

Although the  MIRNS error is a good measure for solution accuracy, it has limitations due to the weighted norm computation, and because a practical metric for solution accuracy can be problem- and designer-dependent. For this example in particular, what really matters would be the differences between the target state values at terminal time and the ones achieved. Figure~\ref{fig:CartPole_Sol_Comp_HS8_PF}  illustrates this aspect by making the sizes of the circles proportional to $\|y_1(t_f)-1,\theta_1(t_f)-\pi,\dot{y}_1(t_f), \dot{\theta}_1(t_f) \|_2$, a measure of terminal constraint violation. With this metric, the value corresponding to the smallest and largest circles shown in the figure is 1.36 and 10.67, respectively. By observing that the sizes of the circles are generally in correspondence with the values of the MIRNS error, we may conclude that, for this example, the MIRNS error is a suitable metric both theoretically and practically.

Figure \ref{fig:CartPole_Sol_Comp_HS8_ResError} illustrates the trends in the reduction of the MIRNS error as the mesh becomes denser. It can be seen that, although the gradient of the lines limited by the degree of the finite element polynomial, DAIR shows a clear advantage in able to obtain solutions with higher accuracy than direct collocation for the same discretization mesh. 

\begin{figure}[tb]
\begin{center}
\includegraphics[width=\columnwidth]{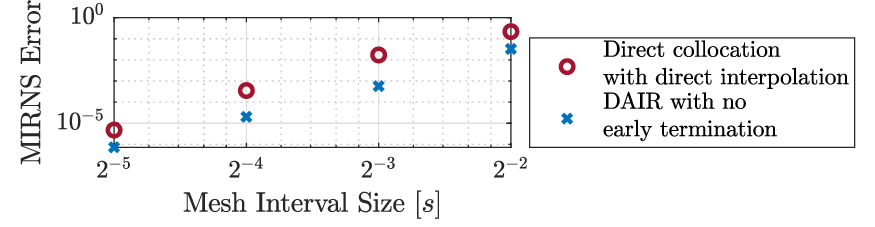}    
\caption{Reduction of MIRNS error as the number of intervals increases (HS/piecewise cubic with equal-spaced intervals).  With early termination, DAIR can in a flexible manner obtain different solutions with errors larger than the blue cross for each mesh design.} 
\label{fig:CartPole_Sol_Comp_HS8_ResError}
\end{center}
\end{figure}

 In Figure~\ref{fig:CartPole_Sol_Comp_HS8_PF}, a comparison of computation time to direct collocation method is also presented. To obtain solutions at the same request error levels, the experimental implementation of DAIR in \texttt{ICLOCS2} saw increases in computation time to different extents depending on the location on the Pareto front. For solutions of relatively high accuracy, most often used in practice, the resulting increase is marginal. It is also important to note that this experimental implementation of DAIR is primarily for proof-of-concept, hence does not yet incorporate code optimizations and detailed exploration of the problem structure and sparsity patterns, which are available in the direct collocation counterpart. With continued developments, it is reasonable to expect integrated residual methods to eventually outperform direct collocation when comparing solutions at the same accuracy level, while maintaining its unique advantages in handling challenging problems.

\section{Conclusions}
\label{sec:Conclusion}
When conventional direct transcription methods, such as direct collocation, are employed to solve nonlinear dynamic optimization problems, assurance in accuracy can only be made a posteriori through error analysis and mesh design iterations. When a given coarse mesh is used, for example in the framework of nonlinear model predictive control, the validity of the solution may become questionable with errors arising inside the intervals between collocation points, despite solving the nonlinear programming problem to negligibly small tolerances. Integrated residual minimization methods fundamentally address this challenge by minimizing the dynamic equation residual error  integrated along the whole trajectory, with the added benefit of being capable of handling difficult problems, such as  those with singular arcs and high-index DAEs. 

Solving DOPs numerically is essentially a multi-objective optimization problem: for a given discretization mesh, one will inevitably face a trade-off between minimizing the objective (for optimality) and minimizing the discretization errors (for accuracy), forming a Pareto front. As demonstrated with the example problems, solutions from direct collocation with a given coarse mesh will be sub-optimal. In contrast, the DAIR scheme is capable of directly obtaining a solution on the Pareto front based on the requested accuracy level. 

Admittedly, the DAIR method is still in an early stage of development. Continued research on these methods will be required in order to realise its full potential and to reach the same level of maturity as direct collocation methods.

\bibliography{main} 
\bibliographystyle{ieeetr}

\end{document}